\documentclass[12pt]{article}
\usepackage{latexsym,amssymb,upref,amsmath,amsthm, amsfonts,authblk}
\usepackage{amssymb,amsmath,amsthm, calc, graphicx}
\usepackage{epsfig}
\usepackage{breqn}
\usepackage{footnpag}
\usepackage{rotating}
\usepackage{amsfonts}
\usepackage{setspace}
\usepackage{fullpage}
\usepackage{enumitem}
\usepackage{bbold}
\usepackage{comment}
\usepackage{pgf,tikz}
\usepackage{mathrsfs}
\usetikzlibrary{arrows}
\usepackage{yhmath}
\usepackage{hyperref}
\usepackage{authblk}
\usepackage{mathrsfs}

\newtheorem{thm}{Theorem}

\newtheorem{proposition}[thm]{Proposition}
\newtheorem{prop}[thm]{Proposition}

\newtheorem{notation}[thm]{Notation}

\newtheorem{rem}[thm]{Remark}

\newcommand{\abs}[1]{\left\lvert{#1}\right\rvert}

\newcommand{\ignore}[1]{}

\title{An improvement on the maximum number of $k$-Dominating Independent Sets}

\begin{document}

\author{
D\'aniel Gerbner $^{a}$
\qquad
Bal\'azs Keszegh $^{a}$
\qquad
Abhishek Methuku $^{b}$\\ 

Bal\'azs Patk\'os $^{a}$
\qquad
M\'at\'e Vizer $^{a}$}

\date{
\today}

\maketitle

\begin{center}
$^a$ Alfr\'ed R\'enyi Institute of Mathematics, HAS, Budapest, Hungary\\
\small H-1053, Budapest, Re\'altanoda utca 13-15.\\
\medskip
$^b$ Central European University, Budapest, Hungary \\
\small H-1051, Budapest, N\'ador utca 9.\\
\medskip
\small \texttt{gerbner,keszegh,patkos@renyi.hu, abhishekmethuku,vizermate@gmail.com}
\medskip
\end{center}

\begin{abstract}

Erd\H{o}s and Moser raised the question of determining the maximum number of maximal cliques or equivalently, the maximum number of maximal independent sets in a graph on $n$ vertices. Since then there has been a lot of research along these lines.

A $k$-dominating independent set is an independent set $D$ such that every vertex not contained in $D$ has at least $k$ neighbours in $D$.
Let $mi_k(n)$ denote the maximum number of $k$-dominating independent sets in a graph on $n$ vertices, and let $\zeta_k:=\lim_{n \rightarrow \infty} \sqrt[n]{mi_k(n)}$. 
Nagy initiated the study of $mi_k(n)$.

In this article we disprove a conjecture of Nagy and prove that for any even $k$ we have $$1.489 \approx \sqrt[9]{36} \le \zeta^k_k.$$

We also prove that for any $k \ge 3$ we have $$\zeta_k^{k} \le 2.053^{\frac{1}{1.053+1/k}}< 1.98,$$
improving the upper bound of Nagy.


\end{abstract}

\vspace{4mm}


\noindent
{\bf AMS Subj.\ Class.\ (2010)}: 05C69

\section{Introduction}

Let $G = G(V, E)$ be a simple graph. For any vertex $v \in V(G)$ let us denote by $d(v)$ the degree of $v$, $N(v)$ denotes the set of neighbors of $v$, also called the open neighborhood of $v$ and $N[v]$ denotes the closed neighborhood, i.e. $N[v] := N(v) \cup \{v\}$.

A subset $I \subset V(G)$ is called \textit{independent} if it does not induce any edges. A \textit{maximal independent} set is an independent set which is not a proper subset of another independent set (that is, it cannot be extended to a bigger independent set). 
A subset $D \subset V(G)$ is a \textit{dominating} set in G if each vertex in $V(G) \setminus D$ is adjacent to at
least one vertex of D, that is, $$\forall v \in  V(G)\setminus D: \ |N(v) \cap D| \ge 1.$$

Erd\H{o}s and Moser raised the question to determine the maximum number of maximal cliques that an $n$-vertex graph might contain. By taking complements, one sees that it is the same as the maximum number of maximal independent sets an $n$-vertex graph can have.  A dominating and independent set $W$ of vertices is often called a \textit{kernel} of the graph (due to Morgenstern and von Neumann \cite{MvN2007}) and clearly, a subset $W$ is a kernel if and only if it is a maximal independent set.

The problem of finding the maximum possible number of kernels has been resolved in many graph families. To state (some of) these results, let $mi_1(n)$ denote the maximum number of maximal independent sets in graphs of order $n$, and let $mi_1(n, \mathcal{F})$ denote the maximum number of maximal independent sets in the $n$-vertex members of the graph
family $\mathcal{F}$. Answering the question of Erd\H{o}s and Moser, Moon and Moser proved the following well known theorem.

\begin{thm}(Moser, Moon, \cite{MM1965})\label{mm} We have

\begin{displaymath}
mi_1(n)=
\left\{ \begin{array}{l l}
3^{n/3} & \textrm{if } n \equiv 0 \ (mod \ 3)\\
\frac{4}{3} \cdot 3^{\lfloor n/3 \rfloor} & \textrm{if } n \equiv 1 \ (mod \ 3)\\
2 \cdot 3^{\lfloor n/3 \rfloor} & \textrm{if } n \equiv 2 \ (mod \ 3)\\
\end{array}
\right.
\end{displaymath}

\end{thm}

Moreover, they obtained the extremal graphs. If addition and multiplication by a positive integer denotes taking vertex disjoint union, then Moser and Moon proved that the equality is attained if and only if the graph $G$ is isomorphic to the graph $n/3$ $K_3$ (if $n \equiv 0$ (mod 3)); to one of the graphs $(\lfloor n/3 \rfloor -1)$ $K_3$ + $K_4$ or $(\lfloor n/3 \rfloor -1)$ $K_3$ + 2 $K_2$ (if $n \equiv 1$ (mod 3)); $\lfloor n/3 \rfloor$ $K_3$ + $K_2$ (if $n \equiv 2$ (mod 3)).

\vskip 0.15truecm
For the family of connected graphs the analogous question was raised by Wilf \cite{W1986} and answered by the following result.

\begin{thm}(F\"uredi \cite{F1987}, Griggs, Grinstead, Guichard \cite{GGG1988}) Let $\mathcal{F}_{con}$ be the family of con-
nected graphs. Then

\begin{displaymath}
mi_1(n, \mathcal{F}_{con})=
\left\{ \begin{array}{l l}
\frac{2}{3} \cdot 3^{n/3} + \frac{1}{2} \cdot 2^{n/3}& \textrm{if } n \equiv 0 \ (mod \ 3)\\
3^{\lfloor n/3 \rfloor} + \frac{1}{2} \cdot 3^{\lfloor n/3 \rfloor}& \textrm{if } n \equiv 1 \ (mod \ 3)\\
\frac{4}{3} \cdot 3^{\lfloor n/3 \rfloor} + \frac{4}{3} \cdot 3^{\lfloor n/3 \rfloor}& \textrm{if } n \equiv 2 \ (mod \ 3)\\
\end{array}
\right.
\end{displaymath}

\end{thm}

The extremal graphs are determined as well. In these graphs, there is a vertex of maximum degree, and its removal yields a member of the extremal graphs list of the previous theorem.

Wilf \cite{W1986} and Sagan \cite{S1988} investigated the case of trees and proved the following theorem.

\begin{thm} Let $\mathcal{T}$ be the family of trees. Then we have

\begin{displaymath}
mi_1(n, \mathcal{T})=
\left\{ \begin{array}{l l}
\frac{1}{2} \cdot 2^{n/2} + 1& \textrm{if } n \equiv 0 \ (mod \ 2)\\
2^{\lfloor n/2 \rfloor} & \textrm{if } n \equiv 1 \ (mod \ 2)\\

\end{array}
\right.
\end{displaymath}

\end{thm}

Hujter and Tuza determined the maximal number of kernels in triangle free graphs by proving the following result.

\begin{thm}(\cite{HT1993}) 
Let $\mathcal{T}_\Delta$ be the family of triangle-free graphs. Then for any integer $n\ge 4$ we have

\begin{displaymath}
mi_1(n, \mathcal{T}_\Delta)=
\left\{ \begin{array}{l l}
2^{n/2} & \textrm{if } n \equiv 0 \ (mod \ 2)\\
5 \cdot 2^{(n-5)/2} & \textrm{if } n \equiv 1 \ (mod \ 2)\\

\end{array}
\right.
\end{displaymath}
\end{thm}

Other related results can be found in the survey of Chang and Jou \cite{CJ1995}.
\vskip 0.3truecm
There are lots of variants of domination studied in the literature. A quite natural and often considered one is $k$-domination. A set $D$ is called \textit{k-dominating} if each
vertex in $V(G) \setminus D$ is adjacent to at least $k$ vertices of $D$. In other words, $$\forall v \in V(G)\setminus D: \ |N(v) \cap D| \ge k.$$

 A $k$-dominating independent set is called a $k$-DIS for short. Note that 1-DISes are exactly maximal independent sets. This notion was introduced by W\l och \cite{W2012}. Nagy \cite{N2017E,N2017} addressed the problem of determining the maximum number of $k$-dominating independent sets (for a given $k \ge 2$) in an $n$-vertex graph. Generalizing $mi_1(n)$ and $mi_1(\mathcal{F})$ we introduce the following notation.

\begin{notation}
For $n,k \ge 1$ let $mi_k(n)$ denote the maximum number of $k$-$DIS$es in graphs of order $n$, and let $mi_k(n, \mathcal{F})$ denote the maximum number of $k$-$DIS$es in an $n$-vertex graph from the family $\mathcal{F}$. If $\mathcal{F}$ consists of a single graph $G$, we denote by $mi_k(G)$ the number of $k$-$DIS$es in $G$.
\end{notation}

In \cite{N2017} Nagy proved that for all $ k \ge 1$ $$\zeta_k:=\lim_{n \rightarrow \infty} \sqrt[n]{mi_k(n)}$$ exists.  
Theorem \ref{mm} implies 
$\zeta_1=\sqrt[3]{3}$ and, by definition, for $k \ge 2$ we have $\zeta_k \in [1, \sqrt[3]{3}].$
The following upper and lower bounds were established on the values of $\zeta_k$.



\vspace{2mm}


\begin{thm}\label{zoligeneral}(Theorem 1.7 \cite{N2017}) For all $k \ge 3$ we have:

$$ \sqrt{2} \le \zeta_k^{k} \le 2^{\frac{k}{k+1}}.$$

\end{thm}

\begin{thm}\label{zoli2}(Theorem 1.6 \cite{N2017}) We have

$$1.489 \approx \sqrt[9]{36} \le \zeta^{2}_2 \le \sqrt[5]{9} \approx 1.551.$$

\end{thm}



Nagy conjectured in \cite{N2017} (Conjecture 2, p19) that the lower bound of Theorem \ref{zoligeneral} will be the value of $\zeta_k^{k}$. Our following theorem disproves this conjecture.

\begin{thm}\label{lower}
For any even $k$ we have $$\sqrt[9]{36} \le \zeta^k_k.$$ Furthermore,  $\lim_\infty \zeta^k_k$ exists and is at least $\sqrt[9]{36}$.
\end{thm}


In this paper, our aim is to show that there is a constant $\eta > 0$ such that $\zeta_k^{k} < 2-\eta$ for all $k \ge 3$, thus improving Theorem \ref{zoligeneral}.


\begin{thm}\label{upper}
For $k \ge 3$ we have $$\zeta_k^{k} \le 2.053^{\frac{1}{1.053+1/k}}< 1.98.$$
\end{thm}

\begin{rem}
It is easy to see that $1.98 < 2^{k/(k+1)}$ for $k \ge 588503$. In fact, the following calculation shows that Theorem \ref{upper} improves Theorem \ref{zoligeneral} for all $k \ge 3$. We want to show that $$2^{k/(k+1)} > (2+ \varepsilon)^{1/(1+\varepsilon +1/k)},$$ for $\varepsilon=0.053$ and any $k \ge 3$. After rearranging we get $$2^\varepsilon > (1+\varepsilon/2)^{1+1/k},$$
which is true for $\varepsilon=0.053$ and $k=3$. Therefore, it is true for any larger $k$.
\end{rem}


\vspace{4mm}

The remainder of the paper is organized as follows. In Section 2 we prove Theorem \ref{lower}, in Section 3 we prove Theorem \ref{upper} and we finish the article with some remarks and open questions in Section 4.

\section{Constructions - Proof of Theorem \ref{lower}}

In this section we gather some observations that are related to lower bound constructions. To be more formal, we introduce the following function: let $m(k,t)$ denote the smallest integer $n$ such that there exists a graph on $n$ vertices that contains at least $t$ $k$-DISes. For our constructions we will need two types of graph products: the \textit{lexicographic product} $G\cdot H$ of two graphs $G$ and $H$ has vertex set $V(G)\times V(H)$ and any two vertices $(u,v)$ and $(x,y)$ are adjacent in $G \cdot H$ if and only if either $u$ is adjacent with $x$ in $G$ or $u=x$ and $v$ is adjacent with $y$ in $H$.

The \textit{cartesian product} $G\times H$ of two graphs $G$ and $H$ also has vertex set $V(G)\times V(H)$ and any two vertices $(u,v)$ and $(x,y)$ are adjacent in $G \cdot H$ if and only if both $u$ is adjacent with $x$ in $G$ and $v$ is adjacent with $y$ in $H$.

All our lower bounds follow from the following remark.

\begin{proposition}\label{mkt}
For any positive integers $k,l,t$, we have

\vspace{2mm}

\textbf{(i)} $mi_k(n)\ge t^{\lfloor \frac{n}{m(k,t)}\rfloor}$, and

\vspace{1mm}

\textbf{(ii)} $m(kl,t)\le lm(k,t)$.
\end{proposition}

\begin{proof}
To prove \textbf{(i)} observe that if $G$ is a graph on $m(k,t)$ vertices containing at least $t$ $k$-DISes, then the graph $G'$ consisting of $\lfloor \frac{n}{m(k,t)}\rfloor$ disjoint copies of $G$ and possibly some isolated vertices, contains at least $t^{\lfloor \frac{n}{m(k,t)}\rfloor}$ many $k$-DISes. Indeed, all isolated vertices must be contained in every $k$-DIS of $G'$, and to form a $k$-DIS of $G'$, one has to pick a $k$-DIS in every copy of $G$.

To prove \textbf{(ii)} let $G$ be a graph on $m(k,t)$ vertices containing at least $t$ $k$-DISes. Then, if we denote by $E_l$ the empty graph on $l$ vertices, the graph $G'=G\cdot E_l$ has $lm(k,t)$ vertices and if $I$ is a $k$-DIS in $G$, then $I'=\{(u,v):u\in I\}$ is a $(kl)$-DIS in $G'$.
\end{proof}

\begin{proof}[Proof of Theorem \ref{lower}] First note (as observed by Nagy already) that $K_3\times K_3$ contains 6 $2$-DISes on 9 vertices. Therefore, by \textbf{(ii)} of Proposition \ref{mkt}, for every even $k$ we have $$m(k,6)\le \frac{k}{2}m(2,6)\le \frac{9k}{2}.$$ Part \textbf{(i)} of Proposition \ref{mkt} yields the statement for even $k$.

\end{proof}

\begin{proposition} $m(k,2)=2k$, $m(k,3)=3k$.

\end{proposition}

\begin{proof} The upper bounds are given by $K_{k,k}$ and $K_{k,k,k}$. For the lower bounds, note that if $A$ and $B$ are two different $k$-DISes, then we have $|A\setminus B|\ge k$ and $|B\setminus A|\ge k$. Indeed, e.g., if $v\in A\setminus B$ then $N(v)$ must contain at least $k$ vertices in $B$, while none of these are in $A$. This observation immediately shows we need at least $2k$ vertices for 2 $k$-DISes. One can easily see by analyzing possible intersection sizes that it also shows we need at least $3k$ vertices for 3 $k$-DISes.
\end{proof}

Note that $K_{k,k,\dots,k}$ gives $m(k,t)\le tk$. Nagy \cite{N2017} showed $m(2,4)=8$ and $m(2,6)= 9$.

\section{Proof of Theorem \ref{upper}}

First of all we fix $k\ge 3$. Let $\varepsilon=0.053$ and choose $c$ such that $$c^{k}=(2 + \varepsilon)^{\frac{1}{1+\varepsilon + 1/k}}.$$

We need to show that $mi_k(n)\le Ac^n$ for some absolute constant $A$. We will proceed by induction on $n$ and the base case is covered by a large enough choice of $A$. Let $G$ be a graph on $n$ vertices containing maximum possible number of $k$-DISes. We assume that every vertex belongs to at least one $k$-DIS, as otherwise we can delete the vertex without decreasing the number of $k$-DISes. Let $v$ be a vertex of minimum degree in $G$ that we denote by $\delta$. Note that we may assume $\delta \ge k$. Indeed, if a vertex $v$ has degree less than $k$, then it is easy to see that it must be contained in every $k$-DIS of $G$. Then it follows that the number of $k$-DISes in $G$ is at most $mi_k(n-\abs{N(v)}-1)$ (where $N(v)$ denotes the set of vertices adjacent to $v$) and we are done by induction.

Consider the following two cases:

\vspace{4mm}

\textbf{Case 1}: $\delta \ge (1 + \varepsilon)k$. 

\vspace{2mm}

In this case we use Proposition 5.1 from \cite{N2017}. Following an inductive argument of F\"uredi \cite{F1987}, Nagy proved that we have 
$$mi_k(n) = mi_k(G)\le c_0\max_{\delta \in \mathbb{Z}^{+}} \{ \left(\frac{k+\delta}{k} \right)^{\frac{n}{\delta+1}} \} .$$ 
for some universal constant $c_0$.  Let $\delta = (1+ \varepsilon')k$. Then we have 

$$mi_k(n) \le c_0(2 + \varepsilon')^{\frac{1}{1+\varepsilon' + 1/k}{\frac{n}{k}}}.$$ 
By Proposition \ref{derivative} (see Appendix), the right hand side of the above inequality is monotone decreasing in $\varepsilon'$.
 Since $\delta \ge (1+\varepsilon)k$, we have $\varepsilon' \ge \varepsilon$. So for fixed $k \ge 3$ we conclude that 
$$mi_k(n) \le c_0(2 + \varepsilon)^{\frac{1}{1+\varepsilon + 1/k}{\frac{n}{k}}} =O(c^{n}).$$

\vspace{4mm}

\textbf{Case 2}: $\delta \le (1 + \varepsilon)k.$

\vspace{2mm}

In this case we combine the inductive argument with a new idea. Let $v$ be a vertex of degree $\delta$. The number of $k$-DISes containing $v$ is at most $mi_k(n-\delta-1)$ and to bound the number of $k$-DISes not containing $v$, we introduce the following auxiliary graph. We say that two non-adjacent vertices $x,y$ of $G$ are  \textit{almost twins} if $$|N(x)\setminus N(y)|, \ |N(y)\setminus N(x)|<k$$ hold. We define $T_G$ to be the graph with vertex set $N(v)$ and $x,y$ form an edge in $T_G$ if they are almost twins in $G$.

\begin{prop}\label{almost}
If $x,y$ belong to the same connected component in $T_G$, then they belong to the same $k$-DISes of $G$. In particular, they are not connected.
\end{prop}

\begin{proof}
It is enough to prove the statement for vertices adjacent in $T_G$. If $x$ belongs to a $k$-DIS $I$ with $y\notin I$, then there should be at least $k$ neighbors of $y$ in $I$ and as $x\in I$, we must have $N(x)\cap I=\emptyset$. This implies $|N(y)\setminus N(x)|\ge k$ which contradicts the fact that $x$ and $y$ are almost twins.
\end{proof}

If a pair of vertices  $x,y\in N(v)$ belong to different components of $T_G$ then the $k$-DISes $I$ containing both of $x$ and $y$ are disjoint from $N(x)\cup N(y)$, and $I\setminus \{x,y\}$ should form a $k$-DIS in $G\setminus (N(x)\cup N(y)\cup \{x,y\})$. As $x$ and $y$ are not almost twins, $|N(x)\cup N(y)|\ge \delta+k$ as wlog. $|N(y)\setminus N(x)|\ge k$ and $|N(x)|\ge k$. Thus, the number of $k$-DISes containing both of $x$ and $y$ is at most $mi_k(n-\delta-k)$. 

On the other hand, if $x$ and $y$ are in the same component $C$ of $T_G$, then by Proposition \ref{almost} any $k$-DIS $I$ containing both of $x$ and $y$ contains all vertices of $C$, is disjoint from $N(C)$ and $I\setminus C$ is a $k$-DIS in $G\setminus (N(C)\cup C)$ and by the second part of Proposition \ref{almost} $N(C)$ and $C$ are disjoint. As $|N(C)|\ge \delta$, the number of $k$-DISes containing both of $x$ and $y$ is at most $mi_k(n-\delta-|C|)$. 

Writing $s_1,s_2,\dots,s_j$ for the sizes of the components of $T_G$, we obtain

\begin{equation}\label{bnd}
mi_k(n)\le mi_k(n-\delta-1)+\frac{\sum_{i=1}^j \binom{s_i}{2}mi_k(n-\delta-s_i)+(\binom{\delta}{2}-\sum_{i=1}^j\binom{s_i}{2})mi_k(n-\delta-k)}{\binom{k}{2}}
\end{equation}
as every $k$-DIS $I$ with $v\notin I$ was counted at least $\binom{k}{2}$ times since $I$ must $k$-dominate $v$.

Let us choose $B=\beta k$ with $\beta=0.8$. This implies $2\le B\le k$ as $k\ge 3$. Suppose that in $T_G$  the union of components of size at most $B$ is $s$. Then the number of pairs of vertices within these components is $\sum_{s_i\le B}\binom{s_i}{2}\le \frac{s(B-1)}{2}$. Also, the number of pairs within components of size larger than $B$ is $\sum_{s_i> B}\binom{s_i}{2}\le \binom{\delta-s}{2}$. Observe that either $s=\delta$ or $s<\delta-B$.


Observe that $mi_k(n-\delta-2)\ge mi_k(n-\delta-B)\ge mi_k(n-\delta-k)$. Thus majoring all $\binom{\delta}{2}$ summands in the following sum we get:	
$$\sum_{i=1}^j\binom{s_i}{2}mi_k(n-\delta-s_i)+(\binom{\delta}{2}-\sum_{i=1}^j\binom{s_i}{2})mi_k(n-\delta-k)\le $$ $$\le \sum_{s_i\le B}\binom{s_i}{2}mi_k(n-\delta-2)+\sum_{s_i> B}\binom{s_i}{2}mi_k(n-\delta-B)+(\binom{\delta}{2}-\sum_{i=1}^j\binom{s_i}{2})mi_k(n-\delta-k)\le $$
$$
\le \frac{s(B-1)}{2} mi_k(n-\delta-2)+\binom{\delta - s}{2}mi_k(n-\delta-B)+\left(\binom{\delta}{2}- \frac{s(B-1)}{2}-\binom{\delta - s}{2}\right) mi_k(n-\delta-k)
$$
As $\binom{\delta}{2}=\frac{s(B-1)}{2}+[s(\delta-s)+\frac{s(s-B)}{2}]+\binom{\delta-s}{2}$, this implies that the right hand side of (\ref{bnd}) is at most  

\begin{equation}\label{bnd2}
mi_k(n-\delta-1)+\frac{s(B-1)}{2\binom{k}{2}}mi_k(n-\delta-2)+\frac{(s(\delta-s)+\frac{s(s-B)}{2})}{\binom{k}{2}}mi_k(n-\delta-k)+\frac{\binom{\delta-s}{2}}{\binom{k}{2}}mi_k(n-\delta-B).
\end{equation}

\noindent
Recall that we want to prove that $mi_k(n)\le Ac^{n}$ for some constant $A$. Using (\ref{bnd2}), by induction after simplifying it would be enough to show  

$$
E:=c^{n} -\left[ c^{n-\delta-1} + \frac{s(B-1)}{2\binom{k}{2}}c^{n-\delta-2} + \frac{(s(\delta-s)+\frac{s(s-B)}{2})}{\binom{k}{2}}c^{n-\delta-k} + \frac{\binom{\delta-s}{2}}{\binom{k}{2}}c^{n-\delta-B}\right]\ge 0.
$$

\noindent



\noindent


\noindent
Using that $k \le \delta$ and simplifying we obtain

\begin{equation}\label{eqbb}
\frac{E}{c^{n-\delta-k}} \ge c^{2k} - \left[c^{k-1}+ \frac{s(B-1)}{k(k-1)}c^{k-2} + \frac{s(2\delta-s-B)}{k(k-1)} + \frac{(\delta-s)(\delta-s-1)}{k(k-1)}c^{k-B}\right].
\end{equation}


\vspace{5mm}

We consider two cases, depending on whether $s$ is equal to $\delta$ or not. In the latter case, $s< \delta-B$, as noted already.

\vspace{2mm}

\textbf{Case 2.1:} $s=\delta$

In this case, the right hand side of (\ref{eqbb}) simplifies to
$$c^{2k} -c^{k-1}- \frac{\delta(B-1)}{k(k-1)}c^{k-2} - \frac{\delta(\delta-B)}{k(k-1)}
.$$

\vspace{2mm}

\noindent

Since $\delta$ is at most $(1+\varepsilon)k$ and replacing $B$ by $\beta k$, the right hand side of the above inequality is at least
\vspace{2mm}


$$c^{2k} - c^{k-1} - (1 +\varepsilon)\frac{(\beta k-1)}{(k-1)}c^{k-2}- (1 + \varepsilon)(1+\varepsilon - \beta)\left(\frac{k}{k-1}\right)
=: f_0(k, \varepsilon, \beta)$$ 
$$\ge c^{2k} - c^k - (1 +\varepsilon)\beta c^k- (1 + \varepsilon)(1+\varepsilon - \beta)(1+\frac{1}{1000}) =: f_1(k, \varepsilon, \beta)$$ for $k>1000$. 


Recall that $\varepsilon=0.053$ and $\beta=0.8$. Note that the function $a^2 - a - (1 +\varepsilon)\beta a- (1 + \varepsilon)(1+\varepsilon - \beta)(1+\frac{1}{1000})$ is increasing in the range $a \ge 1$. At $a = (2 + \varepsilon)^{\frac{1}{1+\varepsilon+1/1000 }}$ the function is positive, thus also for all $k>1000$ at $a = (2 + \varepsilon)^{\frac{1}{1+\varepsilon + 1/k}} = c^k$ the function is positive. This means $f_1(k, \varepsilon, \beta)>0$, which implies $ f_0(k, \varepsilon, \beta) > 0$ for $ k > 1000$. It is easy to check by a simple computer calculation that $ f_0(k, \varepsilon, \beta) > 0$ for $k \le 1000$ as well.


\vspace{5mm}

\textbf{Case 2.2:} $s< \delta-B$.

Note that $\max_{s<\delta -B}\{s(2\delta -s-B)\}< (\delta-B)\delta$. Using this, the right hand side of (\ref{eqbb}) is at least

$$c^{2k} - c^{k-1}- \frac{(\delta-B)(B-1)}{k(k-1)}c^{k-2} - \frac{(\delta-B)\delta}{k(k-1)} - \frac{\delta(\delta-1)}{k(k-1)}c^{k-B}\ge$$
$$c^{2k} - c^k - (1 +\varepsilon-\beta)\frac{(\beta k-1)}{(k-1)}c^k- (1 +\varepsilon-\beta )(1 + \varepsilon )\frac{k}{k-1}$$  $$-\frac{(1+\varepsilon)(k(1+\varepsilon)-1)}{(k-1)}c^{k-\beta k}:=f_2(k,\varepsilon,\beta)$$
$$ \ge c^{2k} - c^k- (1 +\varepsilon-\beta)\beta c^k- (1 +\varepsilon-\beta)(1 + \varepsilon+2/1000)-(1+\varepsilon)(1+\varepsilon(1+1/1000))c^{(1-\beta )k}:=f_3(k,\varepsilon,\beta)$$
for $k>1000$. In the last inequality for bounding the third term we used that $ 2/1000\ge (1+\varepsilon)/(k-1)$ for $k>1000$ as $\varepsilon= 0.053$.

Recall that $\beta=0.8$ and so $1-\beta=\frac{1}{5}$.
Observe that the function $a^{10}-a^5-(1 +\varepsilon-\beta)\beta a^5- (1 +\varepsilon-\beta)(1 + \varepsilon+2/1000)-(1+\varepsilon)(1+\varepsilon(1+1/1000))a$ is increasing in $a$ if $a>1$. As for $a = (2 + \varepsilon)^{\frac{0.2}{1+\varepsilon+1/1000000 }}$ the function is positive, also for all $k>1000000$ for the value $a = (2 + \varepsilon)^{\frac{0.2}{1+\varepsilon + 1/k}} = c^{k-\beta k}$ the function is positive. This means $f_3(k, \varepsilon, \beta)>0$, which implies $ f_2(k, \varepsilon, \beta) > 0$ for $ k > 1000000$. It is easy to check by a simple computer calculation that $ f_2(k, \varepsilon, \beta) > 0$ for $k \le 1000000$.





\vspace{5mm}


Since $\varepsilon=0.053$ and $c^k = (2 + \varepsilon)^{\frac{1}{1+\varepsilon + 1/k}} \le (2 + \varepsilon)^{\frac{1}{1+\varepsilon}}$ for any $k \ge 3$, we get $c^{k} \le 1.98$ for any $k \ge 3$, completing the proof of Theorem \ref{upper}.





\qed

\vspace{3mm}











\section*{Acknowledgement}

We are grateful to the MTA Resort Center of Balatonalm\'adi for their hospitality, where this research was initiated during the Workshop on Graph and Hypergraph Domination in June 2017. We would also like to thank D\'aniel Solt\'esz for helping us in the optimization process, and for introducing us to Wolframalpha cloud \cite{W}.

\vspace{2mm}

\noindent
Research of Gerbner and Patk\'os was supported by the J\'anos Bolyai Research Fellowship of the Hungarian Academy of Sciences.

\vspace{2mm}

\noindent
Research of Gerbner, Keszegh, Methuku and Patk\'os was supported by the National Research, Development and Innovation Office -- NKFIH, grant K 116769.

\vspace{2mm}

\noindent
Research of Patk\'os and Vizer was supported by the National Research, Development and Innovation Office -- NKFIH, grant SNN 116095.

\section*{Appendix}

\begin{prop}\label{derivative} Suppose $k \ge 3$ is fixed. Then the function 

$$f(\varepsilon)=(2 + \varepsilon)^{\frac{1}{1+\varepsilon + 1/k}}$$ is monotone decreasing in $\varepsilon$ for $\varepsilon \in  [0, \infty ) $.

\end{prop}

\begin{proof} As $f$ is differentiable, it is enough to prove that the 
derivative of $f$ is not positive.

$$f'(\varepsilon)=\left( e^{\ln (2+\varepsilon)\frac{1}{1+\varepsilon+\frac{1}{k}}} \right)'=(2 + \varepsilon)^{\frac{1}{1+\varepsilon + 1/k}}\left( \frac{1}{(2+\varepsilon)(1+\varepsilon +\frac{1}{k})}- \frac{\ln (2+\varepsilon)}{(1+\varepsilon +\frac{1}{k})^{2}} \right),$$
so as $(2 + \varepsilon)^{\frac{1}{1+\varepsilon + 1/k}} \ge 0$, it is enough to prove that $$\frac{1}{(2+\varepsilon)(1+\varepsilon +\frac{1}{k})}- \frac{\ln (2+\varepsilon)}{(1+\varepsilon +\frac{1}{k})^{2}} \le 0.$$
Simplifying (and using that $1+\varepsilon +\frac{1}{k} \ge 0$ and $2+\varepsilon \ge 0$), we get
$$1+\varepsilon +\frac{1}{k} \le (2+\varepsilon) \ln (2+\varepsilon).$$

it is easy to check that for $\varepsilon=0$ the above inequality holds as $k \ge 3$. Now note that the derivative of the right hand side with respect to $\varepsilon$, namely $1+\ln(2+\varepsilon)$, is larger than the derivative of the left hand side, namely $1$. Therefore the above inequality holds for all $\varepsilon \ge 0$, and we are done.

\end{proof}

\end{document}